\documentclass[11pt]{article}
\usepackage{amsmath, amsthm, amscd, amsfonts, amssymb, graphicx,enumerate,color}

\theoremstyle{definition}
\newtheorem{theorem}{Theorem}[section]

\newtheorem{corollary}[theorem]{Corollary}
\newtheorem{proposition}[theorem]{Proposition}
\newtheorem{remark}[theorem]{Remark}

\begin{document}
\title{Refined Young inequality and its application to divergences}
\author{Shigeru Furuichi$^1$\footnote{E-mail:furuichi@chs.nihon-u.ac.jp} and Nicu\c{s}or Minculete$^2$\footnote{E-mail:minculeten@yahoo.com}\\
$^1${\small Department of Information Science,}\\
{\small College of Humanities and Sciences, Nihon University,}\\
{\small 3-25-40, Sakurajyousui, Setagaya-ku, Tokyo, 156-8550, Japan}\\
$^2${\small Transilvania University of Bra\c{s}ov, Bra\c{s}ov, 500091, Rom{a}nia}\\}
\date{}
\maketitle
{\bf Abstract.} 
We give bounds on the difference between the weighted arithmetic mean and the weighted geometric mean. These imply refined Young inequalities and the reverses of the Young inequality. 
 We also study some properties on the difference between the weighted arithmetic mean and the weighted geometric mean.  
 Applying the newly obtained inequalities, we show some results on the Tsallis divergence, the R\'{e}nyi divergence, the Jeffreys-Tsallis divergence and the Jensen-Shannon-Tsallis divergence.
\vspace{3mm}

\vspace{3mm}

\vspace{3mm}


\section{Introduction}
The Young integral inequality is the source of many basic inequalities. Young \cite{Yon1912} proved the following: suppose that $f:\left[ {0,\infty } \right) \to \left[ {0,\infty } \right)$  is an increasing continuous function such that $f\left( 0 \right) = 0$ and  $\mathop {\lim }\limits_{x \to \infty } f\left( x \right) = \infty $. Then
\begin{equation}\label{sec1_eq01}
ab \le \int\limits_0^a {f\left( x \right)dx + \int\limits_0^b {{f^{ - 1}}\left( x \right)dx} },
\end{equation}
with equality iff $b=f(a)$. Such a gap is often used to define the Fenchel-Legendre divergence in information geometry \cite{MMN2000,N2021}.
For  $f\left( x \right) = {x^{p - 1}}, (p > 1)$,  in inequality \eqref{sec1_eq01}, we deduce the classical Young inequality:
\begin{equation}\label{sec1_eq02}
ab \le \frac{{{a^p}}}{p} + \frac{{{b^q}}}{q},
\end{equation}
for all  $a,b > 0$ and  $p,q > 1$ with  $\frac{1}{p} + \frac{1}{q} = 1$. The equality occurs if and only if $a^p=b^q$. 

Minguzzi \cite{Ming2008}, proved a reverse Young inequality in the following way:
\begin{equation}\label{sec1_eq03}
0 \le \frac{{{a^p}}}{p} + \frac{{{b^q}}}{q} - ab \le \left( {b - {a^{p - 1}}} \right)\left( {{b^{q - 1}} - a} \right),
\end{equation}
for all  $a,b > 0$ and  $p,q > 1$ with  $\frac{1}{p} + \frac{1}{q} = 1$.

The classical Young inequality \eqref{sec1_eq02} is rewitten as
\begin{equation}\label{sec1_eq03a}
a^{1/p}b^{1/q}\le \dfrac{a}{p}+\dfrac{b}{q}
\end{equation}
by putting $a:=a^{1/p}$ and $b:=b^{1/q}$. 
Putting again
$$
a:=\frac{a_j^p}{\sum\limits_{j=1}^na_j^p},\,\,\,\,b:=\frac{b_j^q}{\sum\limits_{j=1}^nb_j^q}
$$
in the inequality \eqref{sec1_eq03a}, we obtain the famous H\"older inequality:
$$
\sum_{j=1}^na_jb_j\le \left(\sum_{j=1}^na_j^p\right)^{1/p}\left(\sum_{j=1}^nb_j^q\right)^{1/q},\,\,\,\,\left(p,q>1,\,\,\frac{1}{p}+\frac{1}{q}=1\right)
$$
for $a_1,\cdots,a_n>0$ and $b_1,\cdots,b_n>0$,
Thus the inequality \eqref{sec1_eq02} is often reformulated as
\begin{equation}\label{sec1_eq04}
{a^p}{b^{1 - p}} \le pa + \left( {1 - p} \right)b,\,\,a,b> 0,\,\,0\leq p \leq 1
\end{equation}
by putting $1/p=:p$ (then $1/q=1-p$) in the inequality \eqref{sec1_eq03a}. It is notable that $\alpha$-divergence is related to the difference between the weighted arithmetic mean and the weighted geometric mean \cite{N2020}.
 For  $p=1/2$, we deduce the inequality between the geometric mean and the arithmetic mean, $G\left( {a,b} \right): = \sqrt {ab}  \le \frac{{a + b}}{2} = :A\left( {a,b} \right)$. The Heinz mean \cite[Eq.(3)]{Bha2006}(See also \cite{FGG2019}) is defined as ${H_p}\left( {a,b} \right) = \dfrac{{{a^p}{b^{1 - p}} + {a^{1 - p}}{b^p}}}{2}$ and $G\left( {a,b} \right)  \le {H_p}\left( {a,b} \right) \le A\left( {a,b} \right)$. 
 
 Especially, when we discuss about the Young inequality, we will refer to the last form. We consider the following expression
\begin{equation}\label{sec1_eq_add01}
d_p(a,b): = pa + \left( {1 - p} \right)b - {a^p}{b^{1 - p}}
\end{equation}
 which implies that $d_p(a,b) \ge 0$ and $d_p(a,a) = d_0(a,b) = d_1(a,b) = 0$. We remark the following properties: 
 $$d_p(a,b)=b\cdot d_p\left(\frac{a}{b},1\right),\,\,d_p(a,b)=d_{1-p}(b,a),\,\,d_p\left(\frac{1}{a},\frac{1}{b}\right)=\frac{1}{ab}\cdot d_p(b,a).$$
 
 Cartwright-Field inequality (see e.g. \cite{CF1978}) is often written as follows:
\begin{equation}\label{sec1_eq05}
\frac{1}{2}p\left( {1 - p} \right)\frac{{{{\left( {a - b} \right)}^2}}}{{\max \{ a,b\} }} \le d_p(a,b) \le \frac{1}{2}p\left( {1 - p} \right)\frac{{{{\left( {a - b} \right)}^2}}}{{\min \{ a,b\} }}
\end{equation}
for $a,b> 0$ and $0\leq p \leq 1$.
This double inequality gives an improvement of the Young inequality and, at the same time, gives a reverse inequality for the Young inequality.

Kober proved in \cite{Kob1958} a general result related to an improvement of the inequality between arithmetic and geometric means, which for $n = 2$ implies the inequality:
\begin{equation}\label{sec1_eq06}
r{\left( {\sqrt a  - \sqrt b } \right)^2} \le d_p(a,b) \le \left( {1 - r} \right){\left( {\sqrt a  - \sqrt b } \right)^2}
\end{equation}
where  $a,b> 0$, $0\leq p \leq 1$ and $r = \min \left\{ {p,1 - p} \right\}$. This inequality was rediscovered by Kittaneh and Manasrah in \cite{KM2010}. (See also \cite{BK1981}.)

Finally, we found, in \cite{Minc2011}, another improvement of the Young inequality and a reverse inequality, given as:
\begin{equation}\label{sec1_eq07}
r{\left( {\sqrt a  - \sqrt b } \right)^2} + A\left( p \right){\log ^2}\left( {\frac{a}{b}} \right) \le d_p(a,b) \le \left( {1 - r} \right){\left( {\sqrt a  - \sqrt b } \right)^2} + B\left( p \right){\log ^2}\left( {\frac{a}{b}} \right)
\end{equation}
where  $a,b \ge 1$, $0< p < 1$ and $r = \min \left\{ {p,1 - p} \right\}$ with
$A\left( p \right) = \frac{{p\left( {1 - p} \right)}}{2} - \frac{r}{4},B\left( p \right) = \frac{{p\left( {1 - p} \right)}}{2} - \frac{{1 - r}}{4}$.
It is remarkable that the inequalities \eqref{sec1_eq07} give a further refinement of \eqref{sec1_eq06}, since $A(p) \geq 0$ and $B(p)\le 0$.

In \cite{FM2011}, we also presented two inequalities which give two different reverse inequalities for the Young inequality:
\begin{equation}\label{sec1_eq08}
0 \le d_p(a,b) \le {a^p}{b^{1 - p}}\exp \left\{ {\frac{{p\left( {1 - p} \right){{\left( {a - b} \right)}^2}}}{{{{\min }^2}\{ a,b\} }}} \right\} - {a^p}{b^{1 - p}}
\end{equation}
and
\begin{equation}\label{sec1_eq09}
0 \le d_p(a,b) \le p\left( {1 - p} \right){\log ^2}\left( {\frac{a}{b}} \right)\max \{ a,b\} 
\end{equation}
where  $a,b> 0$, $0\leq p \leq 1$. 
See \cite[Chapter 2]{FM2020} for recent advances on refinements and reverses of the Young inequality.

The $\alpha$-divergence is related to the difference of a weighted arithmetic mean with a geometric mean  \cite{N2020}. We mention that the gap is used information geometry to define the Fenchel-Legendre divergence \cite{MMN2000},\cite{N2021}. We give bounds on the difference between the weighted arithmetic mean and the weighted geometric mean. These imply refined Young inequalities and the reverses of the Young inequality. We also study some properties on the difference between the weighted arithmetic mean and the weighted geometric mean. Applying the newly obtained inequalities, we show some results on the Tsallis divergence, the Rényi divergence, the Jeffreys-Tsallis divergence and the Jensen-Shannon-Tsallis divergence \cite{Lin1991}, \cite{Sib1969}. The  parametric Jensen-Shannon  divergence  can  be  used  to  detect  unusual  data,  and  that  one  can  use  it  also  as  a  means  to  perform  relevant  analysis  of  fire  experiments \cite{MAN2020}.

\section{Main results}
We give estimates on $d_p(a,b)$ and also study the properties  of $d_p(a,b)$.
We give the following estimates of  $d_p(a,b)$, firstly.
\begin{theorem}\label{sec2_theorem_1}
For $0<a,b\leq 1$ and $0\leq p\le 1$, we have
\begin{equation}\label{sec2_eq01_1}
r{\left( {\sqrt a  - \sqrt b } \right)^2} + A\left( p \right)ab\cdot{\log ^2}\left( {\frac{a}{b}} \right) \le d_p(a,b) \le \left( {1 - r} \right){\left( {\sqrt a  - \sqrt b } \right)^2} + B\left( p \right)ab\cdot{\log ^2}\left( {\frac{a}{b}} \right)
\end{equation}
where $r = \min \left\{ {p,1 - p} \right\}$ and
$A\left( p \right) = \frac{{p\left( {1 - p} \right)}}{2} - \frac{r}{4},B\left( p \right) = \frac{{p\left( {1 - p} \right)}}{2} - \frac{{1 - r}}{4}$.
\end{theorem}
\begin{proof}
For $p=0$ or $p=1$ or $a=b$, we have equality. We assume $a\neq b$ and $0<p<1$.
Because $0<a,b\le 1$, we have $\frac{1}{a},\frac{1}{b}\ge 1$, so, applying inequality \eqref{sec1_eq07}, we deduce the following relation:
\begin{equation}\label{sec2_eq02_1}
r{\left( {\frac{1}{\sqrt a}  - \frac{1}{\sqrt b} } \right)^2} + A\left( p \right){\log ^2}\left( {\frac{b}{a}} \right) \le d_p\left(\frac{1}{a},\frac{1}{b}\right) \le \left( {1 - r} \right){\left( {\frac{1}{\sqrt a}  - \frac{1}{\sqrt b}} \right)^2} + B\left( p \right){\log ^2}\left( {\frac{b}{a}} \right).
\end{equation}
But, we known that  $d_p(\frac{1}{a},\frac{1}{b})=\frac{1}{ab}\cdot d_{1-p}(a,b)$ and if we replace $p$ by $1-p$ in relation \eqref{sec2_eq02_1} and because $A(p)=A(1-p), B(1-p)=B(p)$, then we proved the inequality from the statement.
\end{proof}

\begin{theorem}\label{sec2_theorem00}
For $a \geq b >0$ and $0<p\le 1$, we have
$$
\frac{p(a-b)(a^{1-p}-b^{1-p})}{2a^{1-p}}\le d_p(a,b)\le  \frac{p(a-b)(a^{1-p}-b^{1-p})}{a^{1-p}}.
$$
\end{theorem}
\begin{proof}
For $p=1$ or $a=b$, we have equality. We assume $a>b$ and $0<p<1$.
It is easy to see that
\begin{equation}\label{sec2_eq00}
\int_1^x(1-t^{p-1})dt=x-1-\frac{x^p-1}{p}.
\end{equation}
We take $x=a/b$ in \eqref{sec2_eq00} and then obtain
$$
pb\int_1^{a/b}(1-t^{p-1})dt=d_p(a,b),\,\,0<p<1
$$

Next, we take the function $f:\left[1,a/b\right]\to \mathbb{R}$ defined by $f(t):=1-t^{p-1}$. By simple calculations we have
$$
\frac{df(t)}{dt}=(1-p)t^{p-2} \ge 0,\,\,\frac{d^2f(t)}{dt^2} =(1-p)(p-2)t^{p-3} \le 0.
$$
So the function $f$ is concave so that we can apply Hermite-Hadamard inequality \cite{NP2018}:
$$
\frac{1}{2}\left(f(1)+f(a/b)\right)\le \frac{1}{a/b-1}\int_1^{a/b}(1-t^{p-1})dt\le f\left(\frac{1+a/b}{2}\right).
$$
The left hand side of the inequalities above shows
$$
\frac{p(a-b)(a^{1-p}-b^{1-p})}{2a^{1-p}}\le d_p(a,b).
$$
Since the function $f(t):=1-t^{1-p}$ is increasing, we have
$$
1-t^{p-1} \le 1-x^{p-1},\,\,(t \le x,\,\,0<p<1).
$$
Integrating the above inequality by $t$ from $1$ to $x$, we get
$$
\int_1^x(1-t^{p-1})dt \le (x-1)(1-x^{p-1})
$$
which implies
$$
d_p(a,b)=bp\int_1^{a/b}(1-t^{p-1})dt \le bp(a/b-1)\left(1-(a/b)^{p-1}\right)
=\frac{p(a-b)(a^{1-p}-b^{1-p})}{a^{1-p}}.
$$
\end{proof}

\begin{theorem}\label{sec2_theorem01}
For $a,b> 0$ and $0\leq p \leq 1$, we have
\begin{equation}\label{sec2_eq01}
p\left( {1 - p} \right)\frac{{{{\left( {a - b} \right)}^2}}}{{\max \{ a,b\} }} \le d_p(a,b) + d_{1-p}(a,b) \le p\left( {1 - p} \right)\frac{{{{\left( {a - b} \right)}^2}}}{{\min \{ a,b\} }}
\end{equation}
\end{theorem}

\begin{proof}
We give two different proofs (I) and (II).
\begin{itemize}
\item[(I)]
For  $a=b$ or $p \in \{ 0,1\} $, we obtain equality in the relation from the statement. Thus, we assume $a \ne b$  and $p \in \left( {0,1} \right)$. 
 It is easy to see that 
 $d_p(a,b) + d_{1-p}(a,b) = a + b - {a^p}{b^{1 - p}} - {a^{1 - p}}{b^p} = ({a^p} - {b^p})({a^{1 - p}} - {b^{1 - p}})$. Using the Lagrange theorem, there exists $c_1$ and $c_2$ between $a$ and $b$ such that $({a^p} - {b^p})({a^{1 - p}} - {b^{1 - p}}) = p(1 - p){(a - b)^2}c_1^{p - 1}c_2^{ - p}$.  But, we have the inequality $\frac{1}{{\max \{ a,b\} }} \le \frac{1}{{c_1^{1 - p}c_2^p}} \le \frac{1}{{\min \{ a,b\} }}$. Therefore, we deduce the inequality of the statement.
 
\item[(II)]Using the Cartwright-Field inequality, we have:
\begin{equation*}
\frac{1}{2}p\left( {1 - p} \right)\frac{{{{\left( {a - b} \right)}^2}}}{{\max \{ a,b\} }} \le d_p(a,b) \le \frac{1}{2}p\left( {1 - p} \right)\frac{{{{\left( {a - b} \right)}^2}}}{{\min \{ a,b\} }}
\end{equation*}
and if we replace $p$ by $1-p$, we deduce 
\begin{equation*}
\frac{1}{2}p\left( {1 - p} \right)\frac{{{{\left( {a - b} \right)}^2}}}{{\max \{ a,b\} }} \le d_{1-p}(a,b) \le \frac{1}{2}p\left( {1 - p} \right)\frac{{{{\left( {a - b} \right)}^2}}}{{\min \{ a,b\} }}
\end{equation*}
for $a,b> 0$ and $0\leq p \leq 1$. By summing up these inequalities, we proved the inequality of the statement
\end{itemize}
\end{proof}
 
 \begin{remark}
 \begin{itemize}
 \item[(i)] From the proof of Theorem \ref{sec2_theorem01}, we obtain 
 $A(a,b) - {H_p}(a,b) = \frac{{d_p(a,b) + d_{1-p}(a,b)}}{2}$, we deduce an estimation for the Heinz mean:
\begin{equation}\label{sec2_eq02}A(a,b) - \frac{1}{2}p\left( {1 - p} \right)\frac{{{{\left( {a - b} \right)}^2}}}{{\min \{ a,b\} }} \le {H_p}(a,b) \le A(a,b) - \frac{1}{2}p\left( {1 - p} \right)\frac{{{{\left( {a - b} \right)}^2}}}{{\max \{ a,b\} }}.
  \end{equation}
 \item[(ii)] Since $d_p(a,b) + d_{1-p}(a,b) = ({a^p} - {b^p})({a^{1 - p}} - {b^{1 - p}})$ and $d_{1-p}(a,b) \ge 0$, we have  $0 \le d_p(a,b) \le ({a^p} - {b^p})({a^{1 - p}} - {b^{1 - p}})$ which  is in fact the inequality given by Minguzzi \eqref{sec1_eq03}.
 \end{itemize}
 \end{remark}
 
 \begin{theorem}\label{sec2_theorem02}
 Let $a,b> 0$ and $0\leq p \leq 1$.
  \begin{itemize}
   \item[(i)] For $1/2 \le p \le 1,\,a \ge b$ or $0 \le p \le 1/2,\,\,a \le b$, we have $d_p(a,b) \ge d_{1-p}(a,b)$.
    \item[(ii)] For $0 \le p \le 1/2,\,a \ge b$ or $1/2 \le p \le 1,\,\,a \le b$, we have $d_p(a,b) \le d_{1-p}(a,b)$.
   \end{itemize}
 \end{theorem}
 \begin{proof}
  For  $a=b$ or $p \in \{ 0,1\} $,  we obtain equality in the relations from the statement. Thus, we assume $a \ne b$  and $p \in \left( {0,1} \right)$. But, we have 
\begin{eqnarray*}
d_p(a,b) - d_{1-p}(a,b) &=& \left( {2p - 1} \right)\left( {a - b} \right) - {a^p}{b^{1 - p}} + {a^{1 - p}}{b^p} \\
&=& b\left( {\left( {2p - 1} \right)\left( {\frac{a}{b} - 1} \right) - {{\left( {\frac{a}{b}} \right)}^p} + {{\left( {\frac{a}{b}} \right)}^{1 - p}}} \right).
\end{eqnarray*}
  We consider the function  $f:(0,\infty ) \to \mathbb{R}$ defined by  $f(t) = (2p - 1)(t - 1) - {t^p} + {t^{1 - p}}$. We calculate the derivatives of $f$, thus we have
  \begin{eqnarray*} 
  &&\frac{{df(t)}}{{dt}} = (2p - 1) - p{t^{p - 1}} + (1 - p){t^{ - p}},\\
  &&\frac{{{d^2}f(t)}}{{d{t^2}}} = (1 - p)p{t^{p - 2}} - p(1 - p){t^{ - p - 1}} = p(1 - p){t^{ - p - 1}}\left( {{t^{2p - 1}} - 1} \right).
  \end{eqnarray*}
  For  $t>1$ and  $1/2\le p <1$, we have $\frac{{{d^2}f(t)}}{{d{t^2}}} > 0$, so,  function  $\frac{{df}}{{dt}}$ is increasing, so we obtain    $\frac{{df\left( t \right)}}{{dt}} > \frac{{df\left( 1 \right)}}{{dt}} = 0$, which implies that function $f$ is increasing, so we have
$f(t) > f(1) = 0$, which means that  $(2p - 1)(t - 1) - {t^p} + {t^{1 - p}} > 0$. For  $t = a/b > 1$, we find that  $d_p(a,b) > d_{1-p}(a,b)$.
 For  $t<1$ and  $0<p\le 1/2$, we have $\frac{{{d^2}f(t)}}{{d{t^2}}} > 0$, so, function $\frac{{df}}{{dt}}$  is increasing, so we obtain    $\frac{{df\left( t \right)}}{{dt}} < \frac{{df\left( 1 \right)}}{{dt}} = 0$, which implies that function $f$ is decreasing, so we have  $f(t) > f(1) = 0$, which means that  $(2p - 1)(t - 1) - {t^p} + {t^{1 - p}} > 0$. For  $t = a/b < 1$, we find that  $d_p(a,b) > d_{1-p}(a,b)$. In the analogous way, we show the inequality in (ii).
  \end{proof}
  
  \begin{remark}
  From (i) in Theorem \ref{sec2_theorem02} for $1/2\le p \le 1$ and $a \ge b$, we have  $d_p(a,b) \ge d_{1-p}(a,b)$, so we obtain 
\begin{equation}\label{sec2_eq03}
\frac{1}{2}p\left( {1 - p} \right)\frac{{{{\left( {a - b} \right)}^2}}}{{\max \{ a,b\} }} \le d_p(a,b),
\end{equation}
which is just left hand side of Cartwright-Field inequality:
$$\frac{1}{2}p\left( {1 - p} \right)\frac{{{{\left( {a - b} \right)}^2}}}{{\max \{ a,b\} }} \le d_p(a,b) \le \frac{1}{2}p\left( {1 - p} \right)\frac{{{{\left( {a - b} \right)}^2}}}{{\min \{ a,b\} }},\,\,(a,b>0,\,\,0\le p \le 1).$$

Therefore, it is quite natural to consider the following inequality
$$d_p(a,b) \ge \frac{1}{2}\left\{ {\frac{1}{2}p\left( {1 - p} \right)\frac{{{{\left( {a - b} \right)}^2}}}{{\max \{ a,b\} }} + \frac{1}{2}p\left( {1 - p} \right)\frac{{{{\left( {a - b} \right)}^2}}}{{\min \{ a,b\} }}} \right\} = \frac{1}{4}p\left( {1 - p} \right){\left( {a - b} \right)^2}\frac{{a + b}}{{ab}}$$
holds or not for a general case $a,b>0$ and $0\le p \le 1$. However, this inequality does not hold in general. We set the function 
$${h_p}(t) = pt + 1 - p - {t^p} - \frac{{p(1 - p)}}{4}\frac{{{{\left( {t - 1} \right)}^2}\left( {t + 1} \right)}}{t},\,\,\,\,\left( {t > 0,\,\,\,\,0 \le p \le 1} \right).$$
Then we have $h_{0.1}(0.3)\simeq -0.00434315$, $h_{0.1}(0.6)\simeq 0.000199783$ and also $h_{0.9}(1.8)\simeq 0.000352199$, $h_{0.9}(2.6)\simeq -0.00282073$.
  \end{remark}
  
  \begin{theorem}\label{sec2_theorem_5}
For $a,b\geq 1$ and $0\leq p\le 1$, we have
\begin{equation}\label{sec1_eq07_1}
\frac{1}{2}p\left( {1 - p} \right)\frac{{{{\left( {a - b} \right)}^2}}}{{\max \{ a,b\} }} \le \frac{1}{2}E_p\left(a,b\right)\le d_p(a,b),
\end{equation}
where $E_p\left(a,b\right):=\min\left\{\frac{p(a-b)(a^{1-p}-b^{1-p})}{\left(\max \{ a,b\}\right)^{1-p}},\frac{(1-p)(a-b)(a^{p}-b^{p})}{\left(\max \{ a,b\}\right)^{p}}\right\}=E_{1-p}\left(a,b\right)$.
\end{theorem}
\begin{proof}
For $p=0$ or $p=1$ or $a=b$, we have equality. We assume $a\neq b$ and $0<p<1$.
If $b<a$, then using Theorem \ref{sec2_theorem00}, we have
\begin{equation*}
\frac{p(a-b)(a^{1-p}-b^{1-p})}{2a^{1-p}}\le d_p(a,b).
\end{equation*}
Using the Lagrange theorem, we obtain $a^{1-p}-b^{1-p}=(1-p)(a-b)\phi^{-p}$, where $b<\phi<a$. For $b\geq 1$, we deduce $a^{1-p}-b^{1-p}\geq(1-p)(a-b)a^{-p}$, which means that $\frac{1}{2}p\left( {1 - p} \right)\frac{{{{\left( {b - a} \right)}^2}}}{a}\le\frac{p(a-b)(a^{1-p}-b^{1-p})}{2a^{1-p}}$. 
If $b>a$ and we replace $p$ by $1-p$, then Theorem \ref{sec2_theorem00} implies
\begin{equation}\label{sec2_eq08_1}
\frac{(1-p)(a-b)(a^{p}-b^{p})}{2b^{p}}\le d_p(a,b).
\end{equation}
Using the Lagrange theorem, we obtain $b^{p}-a^{p}=p(b-a)\theta^{p-1}$, where $a<\theta <b$. For $a\geq 1$, we deduce $b^{p}-a^{p}\geq p(b-a)b^{p-1}$, which means that $\frac{1}{2}p\left( {1 - p} \right)\frac{{{{\left( {b - a} \right)}^2}}}{b}\le\frac{(1-p)(a-b)(a^{p}-b^{p})}{2b^{p}}$. Taking into account the above considerations, we prove the statement. 
\end{proof}

\begin{corollary}\label{sec2_theorem_6}
For $0<a,b\leq 1$ and $0\leq p\le 1$, we have
\begin{equation}\label{sec1_eq07_21}
\frac{1}{2}p\left( {1 - p} \right)\frac{{{{\left( {a - b} \right)}^2}}}{{\max \{ a,b\} }} \le \frac{ab}{2}E_p\left(\frac{1}{a},\frac{1}{b}\right)\le d_p(a,b),
\end{equation}
where $E_{\cdot}(\cdot,\cdot)$ is given in Theorem \ref{sec2_theorem_5}.
\end{corollary}
\begin{proof}
For $p=0$ or $p=1$ or $a=b$, we have the equality. We assume $a\neq b$ and $0<p<1$.
If in inequality \eqref{sec1_eq07_1}, we replace $a,b\leq 1$ by $\frac{1}{a},\frac{1}{b}\geq 1$, we deduce
\begin{equation*}
\frac{1}{2}p\left( {1 - p} \right)\frac{{{{\left( {a - b} \right)}^2}}}{{ab\max \{ a,b\} }} \le \frac{1}{2}E_p\left(\frac{1}{a},\frac{1}{b}\right)\le d_p\left(\frac{1}{a},\frac{1}{b}\right)=\frac{1}{ab}d_p(a,b).
\end{equation*}
 Consequently, we prove the inequalities of the statement. 
\end{proof}

 \begin{theorem}\label{sec2_theorem_7}
For $a,b> 0$ and $0\leq p\le 1$, we have
\begin{equation}\label{sec1_eq07_11}
 d_p(a,b)\leq (1-p)\frac{{{{\left( {a - b} \right)}^2}}}{b}.
\end{equation}
\end{theorem}
\begin{proof}
For $p=0$ or $p=1$ or $a=b$, we have the equality in the relation from the statement. We assume $a\neq b$ and $0<p<1$. We consider function $f:(0, \infty)\to\mathbb{R}$ defined by $f(t)=1-t^{p-1}-(1-p)(t-1)$, $p\in [0,1]$. For $t\in(0,1]$, we have $\frac{df(t)}{dt}=(1-p)(t^{p-2}-1)\geq 0$, which implies that $f$ is increasing, so, we deduce $f(t)\leq f(1)=0$. For $t\in[1,\infty)$, we have $\frac{df(t)}{dt}\leq 0$, which implies that $f$ is decreasing, so, we obtain $f(t)\leq f(1)=0.$ Therefore, we find the following inequality
$$1-t^{p-1}\leq (1-p)(t-1).$$
Multiplying the above inequality by $t>0$, we have 
$$t-t^{p}\leq (1-p)(t^2-t),$$ 
which is equivalent to the inequality $$pt+(1-p)-t^{p}\leq (1-p)(t-1)^2,$$ for all $t>0$ and $p\in [0,1]$.
Therefore, if we take $t=\frac{a}{b}$ in the above inequality and after some calculations, we deduce the inequality of the statement. 
\end{proof}

\begin{corollary}\label{sec2_theorem_8}
For $a,b> 0$ and $0\leq p\le 1$, we have
\begin{equation}\label{sec1_eq07_22}
d_p(a,b)+d_{1-p}(a,b)\leq (1-p)\frac{{{{\left( {a - b} \right)}^2(a+b)}}}{ab}.
\end{equation}
\end{corollary}
\begin{proof}
For $p=0$ or $p=1$ or $a=b$, we have the equality. We assume $a\neq b$ and $0<p<1$.
If in inequality \eqref{sec1_eq07_11}, we exchange $a$ with $b$, we deduce
$$d_p(b,a)\leq (1-p)\frac{{{{\left( {a - b} \right)}^2}}}{a}.$$
But $d_p(b,a)=d_{1-p}(a,b)$, so, we have
$$d_p(a,b)+d_{1-p}(a,b)\leq (1-p)\frac{{{{\left( {a - b} \right)}^2}}}{b}+(1-p)\frac{{{{\left( {a - b} \right)}^2}}}{a}=(1-p)\frac{{{{\left( {a - b} \right)}^2(a+b)}}}{ab}.$$
 Consequently, we prove the inequality of the statement. 
\end{proof}

\section{Applications to some divergences}
The Tsallis divergence (e.g.,\cite{FYK2004,Tsa1998}) is defined for two probability distributions
${\bf p}:=\{p_1,\cdots,p_n\}$ and ${\bf r}:=\{r_1,\cdots,r_n\}$ with $p_j>0$ and $r_j>0$ for all $j=1,\cdots,n$ as
$$
D_q^{T}({\bf p}|{\bf r}) :=\sum_{j=1}^n\frac{p_j-p_j^qr_j^{1-q}}{1-q},\,\,(q>0,\,\,q\ne 1 ).
$$
The R\'{e}nyi divergence (e.g.,\cite{AD1975}) also denoted by
$$
D_q^{R}({\bf p}|{\bf r}) :=\frac{1}{q-1}\log\left(\sum_{j=1}^np_j^{q}r_j^{1-q}\right).
$$
We see in (e.g. \cite{FM2019}) that 
\begin{equation}\label{sec3_eq_add01}
D_q^{R}({\bf p}|{\bf r})=\frac{1}{q-1}\log\left(1+(q-1)D_q^{T}({\bf p}|{\bf r}) \right).
\end{equation}
It is also known that
$$
\lim_{q\to 1}D_q^{T}({\bf p}|{\bf r}) =\lim_{q\to 1}D_q^{R}({\bf p}|{\bf r}) 
= \sum_{j=1}^np_j\log\frac{p_j}{r_j}=:D({\bf p}|{\bf r}),
$$
where $D({\bf p}|{\bf r})$ is the standard divergence (KL information, reltative entropy).
The Jeffreys divergence (see \cite{FM2019}, \cite{FMi2012}) is defined by $J_1({\bf p}|{\bf r}):=D({\bf p}|{\bf r})+D({\bf r}|{\bf p})$
and the Jensen-Shannon divergence \cite{Lin1991,Sib1969} is defined by
 $$JS_1({\bf p}|{\bf r}):=\frac{1}{2}D\left({\bf p}|{\frac{{\bf p+r}}{2}}\right)+\frac{1}{2}D\left({\bf r}|{\frac{{\bf p+r}}{2}}\right).$$
In \cite{MM2013}, the Jeffreys and the Jensen-Shannon divergence are extended to biparametric forms. In \cite{FMi2012}, Furuichi and Mitroi generalizes these divergences to the 
Jeffreys-Tsallis divergence, which is given by $J_q({\bf p}|{\bf r}):=D_q^T({\bf p}|{\bf r})+D_q^T({\bf r}|{\bf p})$
and to the Jensen-Shannon-Tsallis divergence, which is defined as
$$JS_q\left({\bf p}|{\bf r}\right):=\frac{1}{2}D_q^T\left({\bf p}|{\frac{{\bf p+r}}{2}}\right)+\frac{1}{2}D_q^T\left({\bf r}|{\frac{{\bf p+r}}{2}}\right).$$
Several properties of divergences can be extended in the operator theory \cite{MFM}.

For the Tsallis divergence, we have the following relations.
\begin{theorem}\label{sec3_theorem01}
For two probability distributions
${\bf p}:=\{p_1,\cdots,p_n\}$ and ${\bf r}:=\{r_1,\cdots,r_n\}$ with $p_j>0$ and $r_j>0$ for all $j=1,\cdots,n$, we have
\begin{equation}\label{sec3_eq01}
q \sum_{j=1}^n\frac{(p_j-r_j)^2}{\max\{p_j,r_j\}}\le J_q({\bf p}|{\bf r})\le q \sum_{j=1}^n\frac{(p_j-r_j)^2}{\min\{p_j,r_j\}},\,\,(0<q<1 ).
\end{equation} 
\end{theorem}

\begin{proof}
From the definition of the Tsallis divergence, we deduce the equality:
$$
J_q({\bf p}|{\bf r})=\sum_{j=1}^n\frac{p_j+r_j-p_j^qr_j^{1-q}-p_j^{1-q}r_j^q}{1-q}=\frac{1}{1-q}\sum_{j=1}^n\left\{d_q(p_j,r_j)+d_{1-q}(p_j,r_j)\right\},
$$
where $d_{\cdot}(\cdot,\cdot)$ is defined in \eqref{sec1_eq_add01}.
Applying Theorem \ref{sec2_theorem01}, we obtain
$$
q(1-q)\sum_{j=1}^n\frac{(p_j-r_j)^2}{\max\{p_j,r_j\}}\le \sum_{j=1}^n\left\{d_q(p_j,r_j)+d(p_j,r_j)\right\}\le q(1-q)\sum_{j=1}^n\frac{(p_j-r_j)^2}{\min\{p_j,r_j\}}
$$
and combining with the above equality, we deduce the inequalities \eqref{sec3_eq01}.
\end{proof}

\begin{remark}
\begin{itemize}
\item[(i)] In the limit of $q\to 1$ in \eqref{sec3_eq01}, we then obtain
$$
\sum_{j=1}^n\frac{(p_j-r_j)^2}{\max\{p_j,r_j\}}\le J_1({\bf p}|{\bf r})\le \sum_{j=1}^n\frac{(p_j-r_j)^2}{\min\{p_j,r_j\}}
$$
for the standard divergence.
\item[(ii)]From \eqref{sec3_eq_add01}, we have
\begin{eqnarray*}
2+(q-1)\left\{D_q^T({\bf p}|{\bf r})+D_q^T({\bf r}|{\bf p})\right\}&=&
\exp\left((q-1)D_q^R({\bf p}|{\bf r})\right)+\exp\left((q-1)D_q^R({\bf r}|{\bf p})\right)\\
&\ge& 2+(q-1)\left\{D_q^R({\bf p}|{\bf r})+D_q^R({\bf r}|{\bf p})\right\},
\end{eqnarray*}
where we used the inequality $e^x \ge x+1$ for all $x\in\mathbb{R}$.
Thus, we deduce the inequalities:
\begin{equation}\label{remark3.2_ii_ineq1}
D_q^T({\bf p}|{\bf r})+D_q^T({\bf r}|{\bf p}) \le D_q^R({\bf p}|{\bf r})+D_q^R({\bf r}|{\bf p}),\,\,(0<q<1)
\end{equation}
and
$$
D_q^T({\bf p}|{\bf r})+D_q^T({\bf r}|{\bf p}) \ge D_q^R({\bf p}|{\bf r})+D_q^R({\bf r}|{\bf p}),\,\,(q>1).
$$
Combining \eqref{remark3.2_ii_ineq1} with Theorem \ref{sec3_theorem01}, we therefore have the following result for the  R\'{e}nyi divergence 
$$
q\sum_{j=1}^n\frac{(p_j-r_j)^2}{\max\{p_j,r_j\}}\le D_q^R({\bf p}|{\bf r})+D_q^R({\bf r}|{\bf p}),\,\,(0<q<1).
$$
\end{itemize}
\end{remark}
We give the relation between the Jeffreys-Tsallis divergence and the  Jensen-Shannon-Tsallis divergence.
\begin{theorem}\label{sec3_theorem02}
For two probability distributions
${\bf p}:=\{p_1,\cdots,p_n\}$ and ${\bf r}:=\{r_1,\cdots,r_n\}$ with $p_j>0$ and $r_j>0$ for all $j=1,\cdots,n$, we have
\begin{equation}\label{sec3_eq02}
JS_q\left({\bf p}|{\bf r}\right)\leq \frac{1}{4}J_q\left({\bf p}|{\bf r}\right),
\end{equation} 
where $q \ge 0$ with $q \ne 1$. 
\end{theorem}

\begin{proof}
We consider the function $g:(0,\infty)\to\mathbb{R}$ defined by $g(t)=t^{1-q}$, which is concave for $q\in[0,1)$. Therefore, we have $\left(\frac{p_j+r_j}{2}\right)^{1-q}\geq \frac{p_j^{1-q}+r_j^{1-q}}{2}$, which implies the following inequalities $$p_j-p_j^q\left(\frac{p_j+r_j}{2}\right)^{1-q}\leq \frac{p_j-p_j^{q}r_j^{1-q}}{2}, r_j-r_j^q\left(\frac{p_j+r_j}{2}\right)^{1-q}\leq \frac{r_j-r_j^{q}p_j^{1-q}}{2},$$
From the definition of the Tsallis divergence, we deduce the inequality:
$$
D_q^T\left({\bf p}|{\frac{{\bf p+r}}{2}}\right)+D_q^T\left({\bf r}|{\frac{{\bf p+r}}{2}}\right)\leq \frac{1}{2}\left(D_q^T({\bf p}|{\bf r})+D_q^T({\bf r}|{\bf p})\right),
$$
which is equivalent to the relation of the statement. For the case of $q >1$, the function $g(t)=t^{1-q}$ is convex in $t>0$. Similarly, we have the statement, taking into account that $1-q<0$.
\end{proof}
\begin{remark}
In the limit of $q\to 1$ in \eqref{sec3_eq02}, we then obtain
$$
JS_1\left({\bf p}|{\bf r}\right)\leq \frac{1}{4}J_1\left({\bf p}|{\bf r}\right).
$$
\end{remark}
We give the bounds on the Jeffreys-Tsallis divergence by using the refined Young inequality given in Theorem \ref{sec2_theorem_1}. In \cite{LPP}, we found the {\it Battacharyya coefficient} defined as: $$B({\bf p}|{\bf r}):=\sum_{j=1}^{n}{\sqrt {p_jr_j}},$$ which is a measure of the amount of overlapping between two distributions. This can be expressed in terms of the Hellinger distance between the probability distributions
${\bf p}:=\{p_1,\cdots,p_n\}$ and ${\bf r}:=\{r_1,\cdots,r_n\}$, which is given by $$B({\bf p}|{\bf r})=1-h^2({\bf p}|{\bf r}),$$ where the {\it Hellinger distance} (\cite{LPP}, \cite{vEH}) is a metric distance and defined by $$h({\bf p}|{\bf r}):=\frac{1}{\sqrt{2}}\sqrt{\sum_{j=1}^{n}(\sqrt{p_j}-\sqrt{r_j})^2}.$$
\begin{theorem}\label{sec3_theorem03}
For two probability distributions
${\bf p}:=\{p_1,\cdots,p_n\}$ and ${\bf r}:=\{r_1,\cdots,r_n\}$ with $p_j>0$ and $r_j>0$ for all $j=1,\cdots,n$, and $0\leq q<1$, we have
\begin{eqnarray}
&&\frac{4r}{1-q}h^2({\bf p}|{\bf r}) + \frac{2A\left( q \right)}{1-q}\sum_{j=1}^{n}p_jr_j\cdot{\log ^2}\left( {\frac{p_j}{r_j}} \right) \le J_q({\bf p}|{\bf r})\nonumber \\
&& \le \frac{4\left( {1 - r} \right)}{1-q}h^2({\bf p}|{\bf r}) + \frac{2B\left( q \right)}{1-q}\sum_{j=1}^{n}p_jr_j\cdot{\log ^2}\left( {\frac{p_j}{r_j}} \right)\label{sec3_eq03}.
\end{eqnarray}
where $r = \min \left\{ {q,1 - q} \right\}$ and
$A\left( q \right) = \frac{{q\left( {1 - q} \right)}}{2} - \frac{r}{4},B\left( q \right) = \frac{{q\left( {1 - q} \right)}}{2} - \frac{{1 - r}}{4}$.
\end{theorem}
\begin{proof}
For $q=0$, we obtain the equality. Now, we consider $0<q<1$.
Using Theorem \ref{sec2_theorem_1} for $a=p_j<1$ and $b=r_j<1$, $j\in\{1,2,...,n\}$, we deduce
$$
r{\left( {\sqrt p_j  - \sqrt r_j } \right)^2} + A\left( q \right)p_jr_j\cdot{\log ^2}\left( {\frac{p_j}{r_j}} \right) \le d_q(p_j,r_j) \le \left( {1 - r} \right){\left( {\sqrt p_j  - \sqrt r_j } \right)^2} + B\left( q \right)p_jr_j\cdot{\log ^2}\left( {\frac{p_j}{r_j}} \right),
$$
where $r = \min \left\{ {q,1 - q} \right\}$. If we replace $q$ by $1-q$ and taking into account that $A\left( q \right)=A\left( 1-q \right)$ and $B\left( q \right)=B\left( 1-q \right)$, then we have
$$
2r{\left( {\sqrt p_j  - \sqrt r_j } \right)^2} + 2A\left( q \right)p_jr_j\cdot{\log ^2}\left( {\frac{p_j}{r_j}} \right) \le d_q(p_j,r_j)+d_{1-q}(p_j,r_j)  
$$
$$
\le 2\left( {1 - r} \right){\left( {\sqrt p_j  - \sqrt r_j } \right)^2} + 2B\left( q \right)p_jr_j\cdot{\log ^2}\left( {\frac{p_j}{r_j}} \right).
$$
Taking the sum on $j=1,2,\cdots,n$, we find the inequalities
$$
2r\sum_{j=1}^{n}{\left( {\sqrt p_j  - \sqrt r_j } \right)^2} + 2A\left( q \right)\sum_{j=1}^{n}p_jr_j\cdot{\log ^2}\left( {\frac{p_j}{r_j}} \right) \le \sum_{j=1}^{n}\left(d_q(p_j,r_j)+d_{1-q}(p_j,r_j)\right)  
$$
$$
=(1-q)\left(D_q^T({\bf p}|{\bf r})+D_q^T({\bf r}|{\bf p})\right)\le 2\left( {1 - r} \right)\sum_{j=1}^{n}{\left( {\sqrt p_j  - \sqrt r_j } \right)^2} + 2B\left( q \right)\sum_{j=1}^{n}p_jr_j\cdot{\log ^2}\left( {\frac{p_j}{r_j}} \right),
$$
which is equivalent to the inequalities in the statement.
\end{proof}
\begin{remark}
In the limit of $q\to 1$ in \eqref{sec3_eq03}, we then obtain
$$
4h^2({\bf p}|{\bf r})+\frac{1}{2}\sum_{j=1}^{n}p_jr_j\cdot{\log ^2}\left( {\frac{p_j}{r_j}} \right) \le J_1({\bf p}|{\bf r}),$$
since $\lim\limits_{q\to 1}\dfrac{r}{1-q}=1$,  $\lim\limits_{q\to 1}\dfrac{A(q)}{1-q}=\dfrac{1}{4}$ and $\lim\limits_{q\to 1}\dfrac{1-r}{1-q}=\infty$.
\end{remark}
We give the further bounds on the Jeffreys-Tsallis divergence by the use of Theorem \ref{sec2_theorem_5} and Corollary \ref{sec2_theorem_8}.
\begin{theorem}\label{sec3_theorem04}
For two probability distributions
${\bf p}:=\{p_1,\cdots,p_n\}$ and ${\bf r}:=\{r_1,\cdots,r_n\}$ with $p_j>0$ and $r_j>0$ for all $j=1,\cdots,n$, and $0\leq q<1$, we have
\begin{equation}\label{sec3_eq04}
q\sum_{j=1}^{n}\frac{{{{\left( {p_j - r_j} \right)}^2}}}{{\max \{ p_j,r_j\} }} \le \frac{1}{1-q}\sum_{j=1}^{n}p_jr_jE_q\left(\frac{1}{p_j},\frac{1}{r_j}\right)\le J_q({\bf p}|{\bf r})\leq \sum_{j=1}^{n}\frac{{{{\left( {p_j - r_j} \right)}^2(p_j+r_j)}}}{p_jr_j},
\end{equation}
where $E_{\cdot}(\cdot,\cdot)$ is given in Theorem \ref{sec2_theorem_5}.
\end{theorem}
\begin{proof}
Putting $a:=p_j$, $b:=r_j$ and $p:=q$ in \eqref{sec1_eq07_21}, we deduce
\begin{equation*}
\frac{1}{2}q\left( {1 - q} \right)\frac{{{{\left( {p_j - r_j} \right)}^2}}}{{\max \{ p_j,r_j\} }} \le \frac{p_jr_j}{2}E_q\left(\frac{1}{p_j},\frac{1}{r_j}\right)\le d_q(p_j,r_j),
\end{equation*}
and
\begin{equation*}
\frac{1}{2}q\left( {1 - q} \right)\frac{{{{\left( {p_j - r_j} \right)}^2}}}{{\max \{ p_j,r_j\} }} \le \frac{p_jr_j}{2}E_{1-q}\left(\frac{1}{p_j},\frac{1}{r_j}\right)\le d_{1-q}(p_j,r_j).
\end{equation*}
Taking into account that
$$E_q\left(\frac{1}{p_j},\frac{1}{r_j}\right)=E_{1-q}\left(\frac{1}{p_j},\frac{1}{r_j}\right),$$
and by taking the sum on $j=1,2,\cdots,n$, we have 
$$q(1-q)\sum_{j=1}^{n}\frac{{{{\left( {p_j - r_j} \right)}^2}}}{{\max \{ p_j,r_j\} }} \le \sum_{j=1}^{n}p_jr_jE_q\left(\frac{1}{p_j},\frac{1}{r_j}\right)\le \sum_{j=1}^{n}\left(d_q(p_j,r_j)+d_{1-q}(p_j,r_j)\right)$$
we prove the lower bounds of $J_q({\bf p}|{\bf r})$. 
To prove the upper bound of $J_q({\bf p}|{\bf r})$,
we put $a:=p_j$, $b:=r_j$ and $p:=q$ in inequality \eqref{sec1_eq07_22}. 
Then we deduce
$$d_q(p_j,r_j)+d_{1-q}(p_j,r_j)\leq (1-q)\frac{{{{\left( {p_j - r_j} \right)}^2(p_j+r_j)}}}{p_jr_j}.$$
By taking the sum on $j=1,2,\cdots,n$, we find $$\sum_{j=1}^{n}\left(d_q(p_j,r_j)+d_{1-q}(p_j,r_j)\right)\leq (1-q)\sum_{j=1}^{n}\frac{{{{\left( {p_j - r_j} \right)}^2(p_j+r_j)}}}{p_jr_j}.$$
 Consequently, we prove the inequalities of the statement. 
\end{proof}

We also give the further bounds on  the Jeffreys-Tsallis divergence
by the use of  Cartwright-Field inequality given in \eqref{sec1_eq05}.

\begin{theorem}\label{se3_theorem05}
For two probability distributions
${\bf p}:=\{p_1,\cdots,p_n\}$ and ${\bf r}:=\{r_1,\cdots,r_n\}$ with $p_j>0$ and $r_j>0$ for all $j=1,\cdots,n$, and $0\leq q<1$, we have
\begin{eqnarray}
&& \frac{q}{8}\sum_{j=1}^n\left(p_j-r_j\right)^2\left(\frac{1}{p_j+\max\{p_j,r_j\}}+\frac{1}{r_j+\max\{p_j,r_j\}}\right) \nonumber\\
&&\le JS_q({\bf p}|{\bf r})\le \frac{q}{8}\sum_{j=1}^n\left(p_j-r_j\right)^2\left(\frac{1}{p_j+\min\{p_j,r_j\}}+\frac{1}{r_j+\min\{p_j,r_j\}}\right) \label{sec3_eq05}
\end{eqnarray}
\end{theorem}
\begin{proof}
For $q=0$, we have the equality. We assume $0<q<1$. By direct calculations, we have
\begin{eqnarray*}
&&JS_q\left({\bf p}|{\bf r}\right)=\frac{1}{2}D_q^T\left({\bf p}|{\frac{{\bf p+r}}{2}}\right)+\frac{1}{2}D_q^T\left({\bf r}|{\frac{{\bf p+r}}{2}}\right)\\
&&=\frac{1}{2(1-q)}\sum_{j=1}^n\left\{p_j-p_j^q\left(\frac{p_j+r_j}{2}\right)^{1-q}+r_j-r_j^q\left(\frac{p_j+r_j}{2}\right)^{1-q}\right\}\\
&&=\frac{1}{2(1-q)}\sum_{j=1}^n\left\{qp_j+(1-q)\frac{p_j+r_j}{2}-p_j^q\left(\frac{p_j+r_j}{2}\right)^{1-q}+qr_j+(1-q)\frac{p_j+r_j}{2}-r_j^q\left(\frac{p_j+r_j}{2}\right)^{1-q}\right\}\\
&&=\frac{1}{2(1-q)}\sum_{j=1}^n\left\{d_q\left(p_j,\frac{p_j+r_j}{2}\right)+d_q\left(r_j,\frac{p_j+r_j}{2}\right)\right\}.
\end{eqnarray*}
Using inequality  \eqref{sec1_eq05}, we deduce
$$
\frac{q(1-q)}{4}\frac{(p_j-r_j)^2}{p_j+\max\{p_j,r_j\}}\le d_q\left(p_j,\frac{p_j+r_j}{2}\right)
\le \frac{q(1-q)}{4}\frac{(p_j-r_j)^2}{p_j+\min\{p_j,r_j\}}
$$
and
$$
\frac{q(1-q)}{4}\frac{(p_j-r_j)^2}{r_j+\max\{p_j,r_j\}}\le d_q\left(r_j,\frac{p_j+r_j}{2}\right)
\le \frac{q(1-q)}{4}\frac{(p_j-r_j)^2}{r_j+\min\{p_j,r_j\}}.
$$
From the above inequalities we have the statement, by summing on $j=1,2,\cdots,n$.
\end{proof}

It is quite natural to extend the Jensen-Shannon-Tsallis divergence to the following form:
$$
JS_q^v({\bf p}|{\bf r}):=v D_q^T({\bf p}|v{\bf p}+(1-v){\bf r})+(1-v) D_q^T({\bf r}|v{\bf p}+(1-v){\bf r}),\,\, (0\le v\le 1,\,\,q>0,\,\,q\ne 1).
$$
We call this the $v$-weighted Jensen-Shannon-Tsallis divergence. For $v=1/2$, we find that $JS_q^{1/2}({\bf p}|{\bf r})=JS_q({\bf p}|{\bf r})$ which is the Jensen-Shannon-Tsallis divergence. For this quantity $JS_q^v({\bf p}|{\bf r})$, we can obtain the following result in a way similar to the proof of Theorem \ref{se3_theorem05}.

\begin{proposition}
For two probability distributions
${\bf p}:=\{p_1,\cdots,p_n\}$ and ${\bf r}:=\{r_1,\cdots,r_n\}$ with $p_j>0$ and $r_j>0$ for all $j=1,\cdots,n$,  $0\leq q<1$ and $0\le v \le 1$, we have
\begin{eqnarray*}
&& \frac{qv(1-v)}{2}\sum_{j=1}^n\left(p_j-r_j\right)^2\left(\frac{1-v}{v p_j+(1-v)\max\{p_j,r_j\}}+\frac{v}{(1-v) r_j+ v \max\{p_j,r_j\}}\right) \\
&&\le JS_q^v({\bf p}|{\bf r})\le \frac{qv(1-v)}{2}\sum_{j=1}^n\left(p_j-r_j\right)^2\left(\frac{1-v}{v p_j+(1-v)\min\{p_j,r_j\}}+\frac{v}{(1-v) r_j+v \min\{p_j,r_j\}}\right).
\end{eqnarray*}
\end{proposition}
\begin{proof}
We calculate as
\begin{eqnarray*}
&&JS_q^v({\bf p}|{\bf r})=\frac{v}{1-q}\sum_{j=1}^n\left\{p_j-p_j^q\left(vp_j+(1-v)r_j\right)^{1-q}\right\}+\frac{1-v}{1-q}\sum_{j=1}^n\left\{r_j-r_j^q\left(vp_j+(1-v)r_j\right)^{1-q}\right\}\\
&&=\frac{1}{1-q}\sum_{j=1}^n\left\{vp_j+(1-v)r_j-vp_j^q\left(vp_j+(1-v)r_j\right)^{1-q}-(1-v)r_j^q\left(vp_j+(1-v)r_j\right)^{1-q}\right\}\\
&&=\frac{v}{1-q}\sum_{j=1}^n\left\{qp_j+(1-q)\left(vp_j+(1-v)r_j\right)-p_j^q\left(vp_j+(1-v)r_j\right)^{1-q}\right\}\\
&&+\frac{1-v}{1-q}\sum_{j=1}^n\left\{qr_j+(1-q)\left(vp_j+(1-v)r_j\right)-r_j^q\left(vp_j+(1-v)r_j\right)^{1-q}\right\}\\
&&=\frac{1}{1-q}\sum_{j=1}^n\left\{v d_q\left(p_j,vp_j+(1-v)r_j\right)+(1-v) d_q\left(r_j,vp_j+(1-v)r_j\right)\right\}.
\end{eqnarray*}
Using inequality  \eqref{sec1_eq05}, we deduce
$$\frac{q(1-q)}{2}\frac{(1-v)^2(p_j-r_j)^2}{vp_j+(1-v)\max\{p_j,r_j\}} 
\le d_q\left(p_j,vp_j+(1-v)r_j\right)  \le \frac{q(1-q)}{2}\frac{(1-v)^2(p_j-r_j)^2}{vp_j+(1-v)\min\{p_j,r_j\}} 
$$
and
$$\frac{q(1-q)}{2}\frac{v^2(p_j-r_j)^2}{(1-v)r_j+v\max\{p_j,r_j\}} 
\le d_q\left(r_j,vp_j+(1-v)r_j\right)  \le \frac{q(1-q)}{2}\frac{v^2(p_j-r_j)^2}{(1-v)r_j+v\min\{p_j,r_j\}}. 
$$
Multiplying $v$ and $1-v$ by the above inequalities, respectively, and then taking the sum on $j=1,2,\cdots,n$, we obtain the statement.
\end{proof}

\section{Conclusion}

We obtained new inequalities which improve classical Young inequality by analytical calculations with known inequalities.
We also obtained some bounds on the Jeffreys-Tsallis divergence and the Jensen-Shannon-Tsallis divergence. At this point, we do not know clearly that the obtained bounds will play any role in information theory. 
However, if there exsits a purpose to find the meaning of the parameter $q$ in divergences based on Tsallis divergence, then we may state that almost theorems (except for Theorem \ref{sec3_theorem02})  hold for $0\le q <1$. In the first author's previous studies \cite{F2004,F2006}, some results related to Tsallis divergence (relative entropy) are still true for $0\le q <1$, while some results related to Tsallis entropy are still true for $q>1$. In this paper, we treated the Tsallis type divergence so it is shown that almost results are true for $0\le q <1$. This insgiht may give a rough meaning of the parameter $q$.

Since our results in Section 3 are based on the inequalities in Section 2, we summerize on the tightness for our obtained inequalities in Section 2. The double inequality \eqref{sec2_eq01_1} is a counterpart of the double inequality \eqref{sec1_eq07}  for $a,b\in (0,1]$. Therefore they can not be compared each other from the point of view on the tightness, since the conditions are different. The double inequality \eqref{sec2_eq01_1} was used to obtain Theorem \ref{sec3_theorem03}. The double inequality 
 \eqref{sec2_eq01} is essentially Cartwright-Field inequality itself, and it was used to obtain Theorem \ref{sec3_theorem01} as the first result in Section 3. The results in Theorem \ref{sec2_theorem02} are mathematical properties on $d_p(a,b)$.  The inequalities given in \eqref{sec1_eq07_1} gave an improvement of the left hand side in the  inequality \eqref{sec1_eq05} for the case $a,b \ge 1$ and we obtained Theorem \ref{sec3_theorem04} by \eqref{sec1_eq07_1}. We obtained the upper bound of $d_p(a,b)$ as a counterpart of \eqref{sec1_eq07_1} for a general $a,b>0$. This is used to prove Corollary \ref{sec2_theorem_8} which was used to prove Theorem \ref{sec3_theorem04}. However, we find that the upper bound of $d_p(a,b)+d_{1-p}(a,b)$ given in \eqref{sec1_eq07_22} is not tighter than one in \eqref{sec2_eq01}.
 
Finally, Theorem \ref{sec3_theorem02} can be obtained from the  convexity/concavity of the function $t^{1-q}$.  The studies to obtain much sharper bounds will be continued.
We extend the Jensen-Shannon-Tsallis divergence to the following:
$$
JS_q^v({\bf p}|{\bf r}):=v D_q^T({\bf p}|v{\bf p}+(1-v){\bf r})+(1-v) D_q^T({\bf r}|v{\bf p}+(1-v){\bf r}),\,\, (0\le v\le 1,\,\,q>0,\,\,q\ne 1),
$$
and we call this the $v$-weighted Jensen-Shannon-Tsallis divergence. For $v=1/2$, we find that $JS_q^{1/2}({\bf p}|{\bf r})=JS_q({\bf p}|{\bf r})$ which is the Jensen-Shannon-Tsallis divergence. For this quantity as a information-theoretic divergence measure $JS_q^v({\bf p}|{\bf r})$, we obtained several characterizations. 
\section*{Acknowledgements}
The authors would like to thank the referees for their careful and insightful comments to improve our manuscript.
The author (S.F.) was partially supported by JSPS KAKENHI Grant Number 21K03341.

\section*{Author Contributions}
The work presented here was carried out in collaboration between all authors. The study was initiated by the second author. He played also the role of the corresponding author. All authors contributed equally and significantly in writing this article. All authors have read and approved the final manuscript.

\section*{Conflicts of Interest}
The authors declare no conflict of interest.

\end{document}